\documentclass[pdflatex,sn-mathphys-num]{sn-jnl}

\usepackage{graphicx}%
\usepackage{multirow}%
\usepackage{amsmath,amssymb,amsfonts}%
\usepackage{amsthm}%
\usepackage{mathrsfs}%
\usepackage[title]{appendix}%
\usepackage{xcolor}%
\usepackage{textcomp}%
\usepackage{manyfoot}%
\usepackage{booktabs}%
\usepackage{algorithm}%
\usepackage{algorithmicx}%
\usepackage{algpseudocode}%
\usepackage{listings}%

\newtheorem{theorem}{Theorem}
\newtheorem{proposition}[theorem]{Proposition}%

\newtheorem{example}{Example}%

\newtheorem{definition}{Definition}
\newtheorem{lemma}{Lemma}
\newtheorem{corollary}{Corollary}

\counterwithin{definition}{section}
\counterwithin{example}{section}

 \counterwithin{lemma}{section}
\counterwithin{theorem}{section}
\counterwithin{corollary}{section}

\begin{document}

\title[Article Title]{On The Ideals of $\Gamma$-Semigroup}


\author[1]{\fnm{Abin} \sur{Sam Tharakan}}\email{at7105@srmist.edu.in}
\equalcont{These authors contributed equally to this work.}

\author*[1]{\fnm{G.} \sur{Sheeja}}\email{sheejag@srmist.edu.in}
\equalcont{These authors contributed equally to this work.}

\affil[1]{\orgdiv{Department of Mathematics}, \orgname{SRM Institute of Science and Technology}, \orgaddress{\street{Kattankulathur}, \city{Chennai}, \postcode{603203}, \state{Tamil Nadu}, \country{India}}}


\abstract{The concept of $\Gamma$-semigroups was introduced by M. K Sen in 1981. This study aims to investigate several intriguing properties of $\Gamma$-semigroups and to provide the concepts of simple $\Gamma$-semigroups, 0-simple $\Gamma$-semigroups, and completely 0-simple $\Gamma$-semigroups. We prove that non-zero elements of the completely 0-simple $\Gamma$-semigroups form a  $\mathcal{D}$-class and are regular. Fundamental elements of these structures are explored, and we provide concrete results that characterize them using various ideals of $\Gamma$-semigroups and establish the necessary and sufficient condition for a $\Gamma$-semigroups to be completely 0-simple. This study further introduce $\Gamma$-prime ideals and gave some condition in which a $\Gamma$-2-sided ideal to be a $\Gamma$-prime. In addition, we establish a condition for a commutative $\Gamma$ semigroup to be $\Gamma$-prime. we have established how union and intersection of $\Gamma$-prime ideals become $\Gamma$-prime.}




\maketitle
\section{Introduction}

 In 1981, M. K. Sen introduced the concept of a $\Gamma$-semigroup, which extends the ideas of both binary and ternary semigroups\cite{sen1981proceeding}. Subsequent studies by M. K. Sen and N. K. Saha\cite{sen1986semigroup} have explored the fundamental properties and characteristics of $\Gamma$-semigroups. F. M. Sioson's work on the ideal theory in ternary semigroups, conducted in 1965, significantly contributed to and enriched the field\cite{sioson1965ideal}. Further, prime ideals detailed study was done by S. Kar and B. K. Maity \cite{karideals}.
 Green's relations, originally studied by J. A. Green\cite{green1951structure}, are foundational to semigroup theory and have been adapted for $\Gamma$-semigroups by Ronnason Chinram and P. Siammai\cite{article}. The theoretical groundwork for ternary semigroups was laid by M. L. Santiago and Sribala\cite{santiago2010ternary}, who detailed its complexities and mathematical structure, building on the early concepts introduced by Lehmer in 1932\cite{lehmer1932ternary}. Simple, minimal ideals and completely simple semigroups are studied in ternary semigroups by Sribala and G. Sheeja in 2013\cite{sheeja2013simple}.\\
In this study, we explore the concept of $\Gamma$-semigroups. Section 2 presents fundamental definitions essential for this work. In section 3, we further studied the concept of $\Gamma$-simple semigroups, outlining their fundamental properties and defining $\Gamma$-0-simple semigroups, providing characterizations and connected results.\\
Section 4 delves into the concept of o-least $\Gamma$-ideals, presenting theorems related to $\Gamma$-Left Ideals. Section 5 explores the concept of completely 0-simple $\Gamma$-semigroups, demonstrating that such a semigroup is its own $\mathcal{D}$-class and is regular. We further investigate the necessary and sufficient conditions that guarantee a 0-simple semigroup $(\mathcal{T}, \Gamma)$ possesses a completely 0-simple structure. Through rigorous analysis, we establish that the presence of at least one 0-least $\Gamma$-left ideal (0-least $\Gamma$-LId) and one 0-least $\Gamma$-right ideal (0-least $\Gamma$-RId) is both a necessary and sufficient condition for this classification.
In section 6, we introduce the concept of $\Gamma$-prime ideals and provide conditions under which a $\Gamma$-2-sided ideal becomes a $\Gamma$-prime ideal. To illustrate these concepts, we present specific examples that highlight the properties and behavior of $\Gamma$-prime ideals.
Furthermore, we establish a sufficient condition for a commutative $\Gamma$-semigroup to be $\Gamma$-prime. Additionally, counterexamples are provided to demonstrate that the union and intersection of $\Gamma$-prime ideals do not necessarily result in a $\Gamma$-prime ideal. To address this, we also derive conditions under which the union and intersection of $\Gamma$-prime ideals retain the $\Gamma$-prime property.
\section{Preliminaries}

 \begin{definition} \cite{sen1981proceeding}
 Consider $\mathcal{T} \neq \phi$, $\Gamma \neq \phi$ as a two sets. Then $(\mathcal{T},\Gamma)$ is called a $\Gamma$-semigroup if there is a mapping from $\mathcal{T}\times \Gamma \times \mathcal{T}$ to $\mathcal{T}$ such that $(e \alpha f)\beta g=e\alpha (f \beta g)$ for each $e,f,g \in \mathcal{T},  \alpha,\beta \in \Gamma$.
 \end{definition}
 \begin{example}
Let $\mathcal{I}=$Collection of every integer, $\Gamma=$ Collection of every even integer. Consider $\mathcal{I}\times \Gamma \times \mathcal{I}$ as the usual multiplication.Then $(\mathcal{I},\Gamma)$ is a $\Gamma$-semigroup.
\end{example}
\begin{definition}\cite{dutta1993gamma}
$\mathcal{B} \subseteq \mathcal{T}$ in a $\Gamma$-semigroup $(\mathcal{T},\Gamma)$ is known as 
\begin{enumerate}
\item[i.] $\Gamma$-left ideal($\Gamma$-LId) if $[\mathcal{T}\Gamma\mathcal{B}] \subseteq \mathcal{B}$.
\item[ii.] $\Gamma$-right ideal($\Gamma$-RId) if $[\mathcal{B}\Gamma\mathcal{T}] \subseteq \mathcal{B}$.
\item[iii.] $\Gamma$- 2-sided ideal($\Gamma$-2-Id) if $\mathcal{B}$ is $\Gamma$-LId, $\Gamma$-RId.
\end{enumerate}
\end{definition}
\begin{example}
Consider $\mathcal{T}$ as the collection of every $2 \times 3$ matrices whose entries are from integers. Define $\Gamma$ as the collection of every $3 \times 2$ matrices whose entries are from integers.\\
Consider $\mathcal{I}$ as the collection of every $2 \times 3$ matrices whose entries are from positive even integers.
Then $\mathcal{I}$ is an $\Gamma$-2-Id.
\end{example}
\begin{definition}\cite{sen1986semigroup}
$\mathcal{B}\neq \phi \subseteq \mathcal{T}$ is known as $\Gamma$-subsemigroup in a $\Gamma$-semigroup $(\mathcal{T},\Gamma)$ when $[\mathcal{B}\Gamma\mathcal{B}] \subseteq \mathcal{B}$.
\end{definition}
\begin{definition}\cite{sen1986semigroup}
Consider $\mathcal{B} \neq \phi$ as the subset in a $\Gamma$-semigroup $(\mathcal{T},\Gamma)$. Then,
\begin{enumerate}
\item $\Gamma$-LId generated by $\mathcal{B}=\mathcal{B}\cup [\mathcal{T}\Gamma \mathcal{B}]$.
\item $\Gamma$-RId generated by $\mathcal{B}=\mathcal{B}\cup [\mathcal{B}\Gamma \mathcal{T}]$.
\item $\Gamma$-2-sided ideal generated by $\mathcal{B}=\mathcal{B}\cup [\mathcal{T}\Gamma \mathcal{B}]\cup [\mathcal{B}\Gamma \mathcal{T}]\cup [\mathcal{T}\Gamma\mathcal{B}\Gamma \mathcal{T}]$.
\end{enumerate}
\end{definition}
\begin{definition}\cite{sen1986semigroup}
$e \in \mathcal{T}$ is called a $\Gamma$-idempotent in a $\Gamma$-semigroup $(\mathcal{T},\Gamma)$ when for any $a \in \Gamma$, $[eae] = e$.
\end{definition}
\begin{definition}\cite{dutta1993gamma}
An element a is defined to be the regular member in a $\Gamma$-semigroup $(\mathcal{T},\Gamma)$ when there is an $\alpha \in \Gamma$, $[a\alpha a] = a$.
\end{definition}
\begin{definition}\cite{green1951structure}
Let $e,f \in \mathcal{T}$ of a $\Gamma$-semigroup $(\mathcal{T},\Gamma)$. Then, Green's relation is defined as.
\begin{enumerate}
\item $e \mathcal{L} f$ implies and is implied by $[\mathcal{T}\Gamma e]\cup \{e\}=[\mathcal{T}\Gamma f]\cup \{f\}$.
\item $e \mathcal{R} f$ implies and is implied by $[e\Gamma \mathcal{T}]\cup \{e\}=[f\Gamma \mathcal{T}]\cup \{f\}$.
\item $\mathcal{D}=\mathcal{L}\circ \mathcal{R}$.
\end{enumerate}
\end{definition}
\begin{definition}\cite{sen1986semigroup}
A $\Gamma$-semigroup is termed $\Gamma$-right(left) simple when it has no proper $\Gamma$-RId(LId). A $\Gamma$-semigroup is considered a $\Gamma$-2-sided simple if it has no proper $\Gamma$- 2-Id.
\end{definition}
\begin{theorem}\label{thm2.1}\cite{green1951structure}
 e as a regular member in a $\Gamma$-semigroup $(\mathcal{T},\Gamma)$ implies each member of $\mathcal{D}_a$ is also regular.
\end{theorem}
\section{Simple \texorpdfstring{$\Gamma$}{Gamma}-Semigroup}

\begin{lemma}
The statements below are equivalent.
\begin{enumerate}
\item A $\Gamma$-semigroup is left simple.
\item For every element $e$ in $\mathcal{T}$, the set $[\mathcal{T}\Gamma e]$ is equal to $\mathcal{T}$.
\end{enumerate}
\end{lemma}
\begin{corollary}
The necessary and sufficient condition for a $\Gamma$-semigroup to be left $\Gamma$-simple is for any elements $e, f \in \mathcal{T}$, there is $g \in \Gamma$, $h \in \mathcal{T}$ such that $[egh] = f$.
\end{corollary}
\begin{proposition}\label{prop3.1}
A $\Gamma$-semigroup is a  simple $\Gamma$-semigroup that implies and is implied by for every $e \in \mathcal{T}$ $[\mathcal{T}\Gamma e\Gamma \mathcal{T}]=\mathcal{T}$.
\end{proposition}
\begin{proof}
Assume $\Gamma$- semigroup $(\mathcal{T},\Gamma)$ is a simple $\Gamma$-semigroup. Since $(\mathcal{T},\Gamma)$ is $\Gamma$ -LId, $[\mathcal{T}\Gamma e]=\mathcal{T}$, so $[\mathcal{T}\Gamma e]\Gamma \mathcal{T}=[\mathcal{T}\Gamma \mathcal{T}]$. Since $[\mathcal{T}\Gamma \mathcal{T}]\subseteq \mathcal{T}$ and $[\mathcal{T}\Gamma \mathcal{T}]$ is a $\Gamma$-2-sided ideal. So, $[\mathcal{T}\Gamma \mathcal{T}]=\mathcal{T}$. Therefore, $[\mathcal{T}\Gamma e \Gamma \mathcal{T}]=\mathcal{T}$. The opposite part is obvious.
\end{proof}
\begin{corollary}
Consider a $\Gamma$- semigroup as a simple $\Gamma$-semigroup. Then, $[\mathcal{T}\Gamma e]=\mathcal{T}$, $[e \Gamma \mathcal{T}]=\mathcal{T}$ for each $e \in \mathcal{T}$.
\end{corollary}
\begin{example}
    Consider $\mathcal{T}=\{i, -i, 1, -1\}$ and $\Gamma=\{i, -i\}$. $(\mathcal{T}, \Gamma)$ under usual multiplication is a simple $\Gamma$-semigroup.
\end{example}
\begin{definition}
 $u \in \mathcal{T}$ is a zero member denoted by $0$ when $[efu]=[ufe]=u$ for all $e \in \mathcal{T}$ and $f \in \Gamma$.\\
If a $\Gamma$- semigroup $(\mathcal{T},\Gamma)$ doesn't contain 0, Then we can adjoin 0 with $(\mathcal{T},\Gamma)$ by $[efg]=0$ if any of e,g is zero members, $f \in \Gamma$ and $[0f0]=0$ for all $f \in \Gamma$.
Clearly, $\{0\}$ is a $\Gamma$-2 sided ideal of $(\mathcal{T},\Gamma)$.
\end{definition}
\begin{definition}
A $\Gamma$-semigroup $(\mathcal{T},\Gamma)$ that includes zero is known as right(left) 0-simple if $[\mathcal{T}\Gamma\mathcal{T}] \neq 0$ and has no proper $\Gamma$-RId(LId). A $\Gamma$- semigroup $(\mathcal{T},\Gamma)$ 
 is a 0-simple if it possesses no proper $\Gamma$- 2-Id and $[\mathcal{T}\Gamma\mathcal{T}] \neq 0$.
 \end{definition}

\begin{theorem}\label{thm3.1}
The following statements are equivalent for a $\Gamma$-semigroup $(\mathcal{T}, \Gamma)$:
\begin{enumerate}
    \item[i.] $(\mathcal{T}, \Gamma)$ is a 0-simple $\Gamma$-semigroup.
    \item[ii.] For every non-zero $e \in \mathcal{T}$, $[\mathcal{T} \Gamma e \Gamma \mathcal{T}] = \mathcal{T}$.
\end{enumerate}
\end{theorem}

\begin{proof}
\textbf{(i) $\Rightarrow$ (ii):}  
Suppose $(\mathcal{T}, \Gamma)$ is a 0-simple $\Gamma$-semigroup. By definition, $[\mathcal{T} \Gamma \mathcal{T}] \neq 0$.  

Since $[\mathcal{T} \Gamma \mathcal{T}]$ is a $\Gamma$-2-sided ideal, it follows that  
\[
\mathcal{T} = [\mathcal{T} \Gamma \mathcal{T}] = [\mathcal{T} \Gamma \mathcal{T} \Gamma \mathcal{T}].
\]

Now consider any non-zero element $e \in \mathcal{T}$. Since $[\mathcal{T} \Gamma e]$ is a $\Gamma$-left ideal and $[e \Gamma \mathcal{T}]$ is a $\Gamma$-right ideal, we have two possibilities:  
\begin{enumerate}
    \item $[\mathcal{T} \Gamma \mathcal{T} \Gamma \mathcal{T}] = 0$, or  
    \item $[\mathcal{T} \Gamma \mathcal{T} \Gamma \mathcal{T}] = \mathcal{T}$.  
\end{enumerate}

If $[\mathcal{T} \Gamma \mathcal{T} \Gamma \mathcal{T}] = 0$, consider the set  
\[
\mathcal{H} = \{h \in \Gamma : [\mathcal{T} h \mathcal{T}] = 0\}.
\]  
Since $e \neq 0$, $\mathcal{H} \neq 0$ and forms a $\Gamma$-2-sided ideal. By the simplicity of $(\mathcal{T}, \Gamma)$, this implies $\mathcal{H} = \mathcal{T}$.  

However, this would lead to $[\mathcal{T} \Gamma \mathcal{T}] = 0$, which contradicts the assumption that $(\mathcal{T}, \Gamma)$ is 0-simple. Therefore,  
\[
[\mathcal{T} \Gamma \mathcal{T} \Gamma \mathcal{T}] = \mathcal{T} \quad \text{for all non-zero } e \in \mathcal{T}.
\]  

\textbf{(ii) $\Rightarrow$ (i):}  
Conversely, suppose $[\mathcal{T} \Gamma e \Gamma \mathcal{T}] = \mathcal{T}$ for all non-zero $e \in \mathcal{T}$. This implies that any 2-sided ideal generated by a non-zero element is equal to $\mathcal{T}$.  

Thus, $(\mathcal{T}, \Gamma)$ is a 0-simple $\Gamma$-semigroup.  
\end{proof}

 \begin{proposition}
 The $\Gamma$-semigroup  $(\mathcal{T},\Gamma)$ is a 0-simple $\Gamma$-semigroup implies for every $e\neq 0 \in \Gamma$, $[\mathcal{T}e\mathcal{T}]=\mathcal{T}$.
 \end{proposition}
 \begin{proof}
 Assume $\Gamma$-semigroup  $(\mathcal{T},\Gamma)$ is a 0-simple $\Gamma$-semigroup. Then, $[\mathcal{T}\Gamma e] \cup [e \Gamma \mathcal{T}]=\mathcal{T}$ for all non-zero $e \in \mathcal{T}$. Clearly, $[\mathcal{T}e\mathcal{T}]\subseteq \mathcal{T}$ Let $x \in \mathcal{T}$. Then $x=[x_{1}ay_{1}]$.\\
 This implies $x \in [\mathcal{T}e\mathcal{T}]$.\\
 So, $[\mathcal{T}e\mathcal{T}]=\mathcal{T}$ for all $e\neq 0 \in \Gamma$.
 \end{proof}
 \begin{theorem}
     Let $(\mathcal{T},\Gamma)$ be a $\Gamma$-semigroup. A subset  $\mathcal{E} \neq \phi$ of $\mathcal{T}$ is said to be $\Gamma$-2-Id if and only if $[\mathcal{T}\Gamma \mathcal{E} \Gamma \mathcal{T}]\subseteq \mathcal{E}$.
 \end{theorem}
 \begin{proof}
     Assume $\mathcal{E} $ is $\Gamma$-2-Id. Then, $[\mathcal{T}\Gamma \mathcal{E}]\subseteq \mathcal{E}$ and $[\mathcal{E} \Gamma \mathcal{T}]\subseteq \mathcal{E}$. \[ [\mathcal{T}\Gamma \mathcal{E}]\subseteq \mathcal{E} \] 
      \[ [\mathcal{T}\Gamma \mathcal{E}\Gamma \mathcal{T}]\subseteq [\mathcal{E}\Gamma \mathcal{T}] \subseteq \mathcal{E} \]
      Conversly, suppose that $[\mathcal{T}\Gamma \mathcal{E} \Gamma \mathcal{T}]\subseteq \mathcal{E}$. Let $y \in [\mathcal{T}\Gamma \mathcal{E}]$.
      \[ y=t_{1}\alpha_{1}e_{1}\]
      \[ y\alpha_{2}t_{2}=t_{1}\alpha_{1}e_{1}\alpha_{2}t_{2} \in [\mathcal{T}\Gamma \mathcal{E} \Gamma \mathcal{T}] \subseteq \mathcal{E}\]
      \[ t_{3}\alpha_{3}y\alpha_{2}t_{2}\alpha_{4}t_{4} \in [\mathcal{T}\Gamma \mathcal{E} \Gamma \mathcal{T}]\]
      \[ t_{3}\alpha_{3}y\alpha_{2}t_{5} \in [\mathcal{T}\Gamma \mathcal{E} \Gamma \mathcal{T}]\]
      This implies \[ y \in \mathcal{E}\]
      So, \[ [\mathcal{T}\Gamma \mathcal{E}]\subseteq \mathcal{E}\]
      Similarly, we can show that \[ [\mathcal{E} \Gamma \mathcal{T}]\subseteq \mathcal{E}\]
 \end{proof}
 \begin{example}
   Consider $\mathcal{T}=\{e,f,g,h,i\}$ and the binary operation is given in the below table.
    \begin{table}[h!]
\centering
\begin{tabular}{|c|c|c|c|c|c|}
\hline
 & e & f & g & h & i \\ \hline
e & e & e & e & e & e \\ \hline
f & e & e & h & e & f \\ \hline
g & e & e & h & e & g \\ \hline
h & e & e & h & e & h \\ \hline
i & e & f & g & h & i \\ \hline

\end{tabular}
\caption{Operation table}
\label{tab:example}
\end{table}
$\{e,h\}$, $\{e, f, h\}$, $\{e, g, h\}$ and $\{e, f, g, h\}$ are $\Gamma$-2-Id.
\end{example}
-
 \section{0-Least  \texorpdfstring{$\Gamma$}{Gamma}-Ideal}
 \begin{definition}
A nonzero $\Gamma$-left ideal (LId) $\mathcal{B}$ of a $\Gamma$-semigroup $(\mathcal{T}, \Gamma)$ is called a 0-least $\Gamma$-left ideal if $\mathcal{B}$ contains only one $\Gamma$-left ideal, namely $(0)$.

Similarly, we can define a 0-least $\Gamma$-right ideal (RId).

A nonzero $\Gamma$-two-sided ideal $\mathcal{B}$ in a $\Gamma$-semigroup $(\mathcal{T}, \Gamma)$ is called a 0-least $\Gamma$-two-sided ideal if $\mathcal{B}$ contains only one $\Gamma$-two-sided ideal, namely $(0)$.
\end{definition}

 \begin{lemma}\label{lemma4.1}
If $\mathcal{H}$ is the least $\Gamma$-LId of $(\mathcal{T},\Gamma)$, and $e \in \Gamma, f \in \mathcal{T}$, then $[\mathcal{H}ef]$ is the least $\Gamma$-LId of $(\mathcal{T},\Gamma)$.
\end{lemma}

\begin{proof}
We know that $[\mathcal{H}ef]$ is a $\Gamma$-LId of $(\mathcal{T},\Gamma)$.\\
Let $\mathcal{A}$ be a $\Gamma$-LId of $(\mathcal{T},\Gamma)$ included in $[\mathcal{H}ef]$.\\
Define $\mathcal{M} = \{m \in \mathcal{H} : [mef] \in \mathcal{A}\}$.\\
Thus, $[\mathcal{M}ef] = \mathcal{A}$.\\
Let $n_1 \in \mathcal{T}, n_2 \in \Gamma$, and $m \in \mathcal{M}$.\\
Then,
\[
[n_{1}n_{2}m]ef = n_{1}n_{2}[mef] \in \mathcal{T}\Gamma\mathcal{A} \subseteq \mathcal{A}.
\]
Therefore, $\mathcal{M}$ is a $\Gamma$-LId of $(\mathcal{T},\Gamma)$ contained in $\mathcal{H}$.\\
Hence, $\mathcal{M} = \mathcal{H}$.\\
This implies $\mathcal{A} = [\mathcal{H}ef]$, and thus $[\mathcal{H}ef]$ is the least $\Gamma$-LId.
\end{proof}

\begin{theorem}\label{thm4.1}
Let $\mathcal{A}$ be the least $\Gamma$-2-Id of $(\mathcal{T},\Gamma)$. Then $\mathcal{A}$ is a simple $\Gamma$-semigroup.
\end{theorem}

\begin{proof}
We know that $[\mathcal{A}\Gamma\mathcal{A}]$ is a $\Gamma$-2-Id of $(\mathcal{T},\Gamma)$ and $[\mathcal{A}\Gamma\mathcal{A}] \subseteq \mathcal{A}$.\\
Thus, $[\mathcal{A}\Gamma\mathcal{A}] = \mathcal{A}$.\\
For any $a \in \mathcal{A}$, we have:
\[
(a) = \{a\} \cup [\mathcal{T}\Gamma a] \cup [a\Gamma \mathcal{T}] \cup [\mathcal{T}\Gamma a\Gamma \mathcal{T}].
\]
So, $(a) = \mathcal{A}$.\\
Now:
\[
\mathcal{A} = [\mathcal{A}\Gamma\mathcal{A}] = [\mathcal{A}\Gamma\mathcal{A}\Gamma\mathcal{A}] = [\mathcal{A}\Gamma(a)\Gamma\mathcal{A}],
\]
where:
\[
(a) = \{a\} \cup [\mathcal{T}\Gamma a] \cup [a\Gamma \mathcal{T}] \cup [\mathcal{T}\Gamma a\Gamma \mathcal{T}].
\]
Substituting, we get:
\[
\mathcal{A} = [\mathcal{A}\Gamma\{a\} \cup [\mathcal{T}\Gamma a] \cup [a\Gamma \mathcal{T}] \cup [\mathcal{T}\Gamma a\Gamma \mathcal{T}]\Gamma\mathcal{A}] = [\mathcal{A}\Gamma(a)\Gamma\mathcal{A}].
\]
This implies:
\[
[\mathcal{A}\Gamma\mathcal{A}\Gamma\mathcal{A}] = \mathcal{A}.
\]
Therefore, $\mathcal{A}$ is a simple $\Gamma$-semigroup by \textbf{Proposition}~\ref{prop3.1}.
\end{proof}
\begin{align*}
    \end{align*}

\begin{example}
  Let $\mathcal{T}=\{a, b, c, d\}$ and $\Gamma=\{\alpha\}$ . Operation is $a\alpha b=ab$ for every $a, b \in \mathcal{T}$ and the operation table is given below.

\begin{table}[h!]
\centering
\begin{tabular}{|c|c|c|c|c|}
\hline
 & a & b & c & d \\ \hline
a     & a     & a      & d      & d      \\ \hline
b      & a      & a      & d     & d      \\ \hline
c      & d      & d      & a      & a     \\ \hline
d      & d      & d      & a     & a      \\ \hline
\end{tabular}
\caption{Operation table}
\label{tab:example}
\end{table}
Here, $\{a, d\}$ is $\Gamma$-2-Id and $\{a, d\}$ is simple.
\end{example}

\begin{lemma}
$\mathcal{C}$ is the intersection of every $\Gamma$-2-Id implies $\mathcal{C}$ is a 2-sided simple $\Gamma$-semigroup.
\end{lemma}
\begin{proof}
For any $c \in \mathcal{C}$, $(c)=\{c\}\cup [\mathcal{T}\Gamma c]\cup [c\Gamma \mathcal{T}]\cup [\mathcal{T}\Gamma c\Gamma \mathcal{T}]$ is a $\Gamma$ 2-Id contained in $\mathcal{C}$.\\
$(c)=\mathcal{C}$ for all $c \in \mathcal{C}$.\\
Therefore, $\mathcal{C}$ is the unique least 2-Id of $(\mathcal{T},\Gamma)$.\\
So, $\mathcal{C}$ is a 2-sided simple $\Gamma$-semigroup by \textbf{Theorem}\ref{thm4.1}.
\end{proof}
\begin{lemma}\label{lemma4.3}
Consider $\mathcal{H}$ as the 0-least $\Gamma$-LId of $(\mathcal{T},\Gamma)$ with zero and $[\mathcal{H}\Gamma \mathcal{H}] \neq 0$. Then for any non zero element $e \in \mathcal{H}$, $\mathcal{H}=[\mathcal{T}\Gamma e]$.
\end{lemma}
\begin{proof}
Let $e \neq 0 \in \mathcal{H}$.\\
Then, $[\mathcal{T}\Gamma e]$ is a $\Gamma$-LId of $(\mathcal{T},\Gamma)$ included in $\mathcal{H}$.\\
Assume $[\mathcal{T}\Gamma e]=0$.\\
Then, $ [a\Gamma a=0]$, $\{0,a\}$ is a non zero LId of $(\mathcal{T},\Gamma)$ included in $\mathcal{H}$.\\
Then, $\{0,a\}=\mathcal{H}$, $[\mathcal{H}\Gamma \mathcal{H}]=0$, which is a contradiction.\\
S0, $[\mathcal{T}\Gamma e]\neq 0$ and $[\mathcal{T}\Gamma e] = \mathcal{H}$.
\end{proof}
\begin{lemma}
$\mathcal{H}$ is a 0-least $\Gamma$-LId of a
$\Gamma$-semigroup $(\mathcal{T},\Gamma)$ including zero and $e \in \Gamma , f \in \mathcal{T}$ implies $[\mathcal{H}ef]$ is either zero or 0-least $\Gamma$-LId of $(\mathcal{T},\Gamma)$.
\end{lemma}
\begin{proof}
Let $[\mathcal{H}ef] \neq 0$.\\
So, $[\mathcal{H}ef]$ is a  $\Gamma$-LId of $(\mathcal{T},\Gamma)$.\\
Consider $\mathcal{N}$ as the $\Gamma$-LId of $(\mathcal{T},\Gamma)$ included in $[\mathcal{H}ef]$.\\
Take $\mathcal{M}=\{m \in \mathcal{H}: [mef]\in \mathcal{N}\}$.\\
Then, $[\mathcal{M}ef]=\mathcal{N}$.\\
By \textbf{lemma}\ref{lemma4.1}, we can show that $\mathcal{M}$ is a $\Gamma$-LId of a
$\Gamma$-semigroup $(\mathcal{T},\Gamma)$.\\
So, $\mathcal{M}$ is zero or $\mathcal{M}$ is $\mathcal{H}$.\\
Then, $\mathcal{N}=0$ or $\mathcal{N}=[\mathcal{H}ef]$.\\
Therefore, $ [\mathcal{H}ef]$ is a 0-least $\Gamma$-LId of $(\mathcal{T},\Gamma)$.
\end{proof}
\begin{theorem}\label{thm4.2}
$\mathcal{H}$ is a 0-least $\Gamma$-2-Id of $\Gamma$-semigroup $(\mathcal{T},\Gamma)$ including zero implies either $[\mathcal{H}\Gamma \mathcal{H}]=0$ or $\mathcal{H}$ is simple $\Gamma$-semigroup.
\end{theorem}
\begin{proof}
$[\mathcal{H}\Gamma \mathcal{H}]$ is a $\Gamma$-2-sided ideal of $\Gamma$-semigroup $(\mathcal{T},\Gamma)$ and $[\mathcal{H}\Gamma \mathcal{H}] \subseteq \mathcal{H}$.\\
So, either $[\mathcal{H}\Gamma \mathcal{H}]=0$ or $[\mathcal{H}\Gamma \mathcal{H}]=\mathcal{H}$.\\
Let $[\mathcal{H}\Gamma \mathcal{H}]\neq 0$.\\
So, $\mathcal{H}=[\mathcal{H}\Gamma \mathcal{H}]=[\mathcal{H}\Gamma \mathcal{H}\Gamma \mathcal{H}]$.\\
From \textbf{Theorem}\ref{thm3.1}, we can show that for all non zero $e \in \mathcal{H}$, $[\mathcal{H}
\Gamma e \Gamma \mathcal{H}]$.\\
Thus, $\mathcal{H}$ is 0-2-sided simple $\Gamma$-semigroup.
\end{proof}
\begin{theorem}\label{thm4.3}
$\mathcal{H}$ be a 0-least $\Gamma$-2-Id of a $\Gamma$-semigroup $(\mathcal{T},\Gamma)$ contains atleast one 0-least $\Gamma$-LId of $(\mathcal{T},\Gamma)$ implies $\mathcal{H}$ is the union of every 0-least $\Gamma$-LIds of $(\mathcal{T},\Gamma)$ included in $\mathcal{H}$.
\end{theorem}
\begin{proof}
Consider $\mathcal{M}$ represent the union of every 0-least $\Gamma$-LIds of $(\mathcal{T},\Gamma)$ that are included within $\mathcal{H}$.\\
$\mathcal{M}$ is a LId of $(\mathcal{T},\Gamma)$.\\
It demonstrates that $\mathcal{M}$ is a $\Gamma$-RId. Consider $h \in \mathcal{M}$, let $i \in \mathcal{T}$ and $j \in \Gamma$.\\
From the definition, $h \in \mathcal{M}$ for some 0-least $\Gamma$-LId $\mathcal{N}$ of $(\mathcal{T},\Gamma)$ included in$\mathcal{H}$. \\
From \textbf{lemma}\ref{lemma4.1}, $[\mathcal{N}ij] = 0$ or $[\mathcal{N}ij]$ is a 0-least $\Gamma$-LId of $(\mathcal{T},\Gamma)$.\\
Also, $[\mathcal{N}ij] \subseteq [\mathcal{H}ij] \subseteq \mathcal{H}$ and hence $[\mathcal{N}ij] \subseteq \mathcal{M}$.\\
So, for all $h \in \mathcal{L}$, $[\mathcal{L}ij] \subseteq \mathcal{L}$.\\
Hence, $\mathcal{L} \neq 0$ since it includes minimum 1 0-least $\Gamma$-LId of $(\mathcal{T},\Gamma)$.\\
So, $\mathcal{L} \neq \phi$ is a $\Gamma$-2-Id of $(\mathcal{T},\Gamma)$ included in $\mathcal{H}$.\\
Hence, $\mathcal{H}=\mathcal{M}$.
\end{proof}
\begin{lemma}
If $\mathcal{H}$ be a 0-least $\Gamma$-2-sided ideal of
 $\Gamma$-semigroup $(\mathcal{T},\Gamma)$ includes zero and $[\mathcal{H}\Gamma \mathcal{H}] \neq 0$. Then, for any non-zero $\Gamma$-LId $\mathcal{M}$ of $(\mathcal{T},\Gamma)$ included in $\mathcal{H}$, $[\mathcal{M}\Gamma \mathcal{M}] \neq 0$.
 \end{lemma}
 \begin{proof}
 Here, either $[\mathcal{M}\Gamma \mathcal{T}]=\mathcal{M}$ or $[\mathcal{M}\Gamma \mathcal{T}]=0$ as $[\mathcal{M}\Gamma \mathcal{T}]$ is a $\Gamma$-2-Id of $(\mathcal{T},\Gamma)$ included in $\mathcal{H}$.\\
 Suppose $[\mathcal{M}\Gamma \mathcal{T}]=0$.\\
 So, $\mathcal{M}$ is a $\Gamma$-2-Id and $\mathcal{M}=\mathcal{H}$.\\
 So, $[\mathcal{H}\Gamma \mathcal{H}]=[\mathcal{M}\Gamma \mathcal{H}] \subseteq [\mathcal{M}\Gamma \mathcal{T}]=0$, which is a contradiction.\\
 Hence, $[\mathcal{M}\Gamma \mathcal{H}]=\mathcal{H}$.
\\
So, $\mathcal{H}=[\mathcal{H}\Gamma \mathcal{H}]=[\mathcal{M}\Gamma \mathcal{T}\Gamma \mathcal{M}\Gamma \mathcal{T}]\subseteq [\mathcal{M}\Gamma \mathcal{M}]\Gamma \mathcal{T}$.\\
Therefore, $ [\mathcal{M}\Gamma \mathcal{M}]\neq 0$.
\end{proof}
\begin{theorem}
If $\mathcal{H}$ be a 0-least $\Gamma$-2-sided ideal of
 $\Gamma$-semigroup $(\mathcal{T},\Gamma)$ includes zero and $[\mathcal{H}\Gamma \mathcal{H}] \neq 0$ and suppose that $\mathcal{H}$ contains atleast one 0-least $\Gamma$-LId of $(\mathcal{T},\Gamma)$. Then, each $\Gamma$-LId of $\mathcal{H}$ is a $\Gamma$-LId of $(\mathcal{T},\Gamma)$.
 \end{theorem}
 \begin{proof}
 Consider $\mathcal{M}$ as the non zero $\Gamma$-LId of $\mathcal{M}$, a non zero element $a \in \mathcal{M}$.\\
 Then, $\mathcal{H}$ is a 0-2-sided simple by \textbf{Theorem}\ref{thm4.2}.\\
 $\mathcal{H}=[\mathcal{H}\Gamma a \Gamma\mathcal{H}]$.\\
 Hence, $[\mathcal{H}\Gamma a] \neq 0$.\\
 There is a 0-least $\Gamma$-LId $\mathcal{H}_{1}$ of $(\mathcal{T},\Gamma)$ such that $a \in \mathcal{H}_{1} \subseteq \mathcal{H}$ by \textbf{Theorem}\ref{thm4.2}.\\
 Since, $[\mathcal{H}\Gamma a]$ is a non zero $\Gamma$-LId of $(\mathcal{T},\Gamma)$ included in $\mathcal{H}_{1}$.\\
 Then, $[\mathcal{H}\Gamma a]=\mathcal{H}_{1}$.\\
 Therefore, $a \in [\mathcal{H}\Gamma a]$.\\
 Thus, $\mathcal{M}=\bigcup \{[\mathcal{H}\mathcal{H}a]:a \in \mathcal{M}\}$ becomes $\Gamma$-LId of $(\mathcal{T},\Gamma)$.
 \end{proof}
 Similar way we can prove the results for $\Gamma$-RIds.
 \section{Completely 0-Simple \texorpdfstring{$\Gamma$}{Gamma}-Semigroup}
 \begin{definition}
 Consider e,f as the idempotents in a $\Gamma$-Semigroup $(\mathcal{T},\Gamma)$.Then, a partial order $\leq$ can be defined as \\
 $e \leq f$ if there exist $x,y \in \Gamma$, $[exf]=[fye]=e$.\\
 An idempotent is said to be primitive if it is nonzero and is minimal in the set of non-zero idempotents(With the above partial order).\\
 A 0-simple $\Gamma$-semigroup is called complete when it contains a primitive idempotent.
 \end{definition}
 \begin{lemma}\label{lemma5.1}
Let $\mathcal{H}$ be 0-least $\Gamma$-LId of $\Gamma$-semigroup $(\mathcal{T},\Gamma)$ implies $\mathcal{H}\setminus \{0\}$ is a $\mathcal{L}$ class of $(\mathcal{T},\Gamma)$.
\end{lemma}
\begin{proof}
$[\mathcal{T}\Gamma e]$ is a $\Gamma$-LId of $(\mathcal{T},\Gamma)$ included in $\mathcal{H}$ for all $e \in \mathcal{H}$.\\
So, $[\mathcal{T}\Gamma e]=0$ or $[\mathcal{T}\Gamma e]=\mathcal{H}$.\\
Assume that $[\mathcal{T}\Gamma e]=\mathcal{H}$ for every $e \in \mathcal{H}\setminus \{0\}$.\\
Then, $e \cup [\mathcal{T}\Gamma e]=\mathcal{H}=f \cup [\mathcal{T}\Gamma f]$ for all $e,f \in \mathcal{H}\setminus \{0\}$.\\
So, $ \mathcal{H}\setminus \{0\}$ is included in the $\mathcal{L}$-class $\mathcal{L}_{e}$.\\
$f \in \mathcal{L}_{e}$ implies $f \in e \cup [\mathcal{T}\Gamma e]=\mathcal{H}$.\\
So, that $\mathcal{L}_{e} \subseteq \mathcal{H}\setminus \{0\}$.\\
So, $\mathcal{H}\setminus \{0\}$ is a $\mathcal{L}$ class.\\
Assume some $e \in \mathcal{H}$, $[\mathcal{T}\Gamma e]=0$.\\
Thus, $\{0,e\}$ as the nonzero $\Gamma$-LId of $(\mathcal{T},\Gamma)$ included in $\mathcal{H}$.\\
This implies $\{0,e\}=\mathcal{H}$.\\
Thus $e \cup [\mathcal{T}\Gamma e]=\mathcal{H}$ and $e \mathcal{L} f \implies e=f$.\\
Therefore, $\mathcal{H}\setminus \{0\}$ is a $\mathcal{L}$ class of $(\mathcal{T},\Gamma)$.
\end{proof}
Similar way, we can show that for the $\Gamma$-RId.
\begin{theorem}
If $(\mathcal{T},\Gamma)$ is a completely 0-simple $\Gamma$-semigroup. Then, non zero members in $(\mathcal{T},\Gamma)$ forms $\mathcal{D}$ class and $(\mathcal{T},\Gamma)$ is $\Gamma$-regular.
\end{theorem}
\begin{proof}
Consider $(\mathcal{T},\Gamma)$ as a completely 0-2-sided simple $\Gamma$-semigroup.\\
Take nonzero element $e,f \in \mathcal{T}$.\\
Then, e is in some 0-least $\Gamma$-LId $\mathcal{H}$, f is in some 0-least $\Gamma$-RId $\mathcal{I}$.\\
So, $\mathcal{H}=[\mathcal{T}\Gamma e]$ and $\mathcal{I}=[f\Gamma \mathcal{T}]$.\\
By \textbf{Lemma}\ref{lemma5.1}, $\mathcal{H}\setminus \{0\}$ is a $\mathcal{L}$ class of $(\mathcal{T},\Gamma)$ containing e and $\mathcal{I}\setminus \{0\}$ is a $\mathcal{R}$ class of $(\mathcal{T},\Gamma)$ containing f.\\
Thus, $[f\Gamma e] \subseteq \mathcal{H}_{e} \cap \mathcal{I}_{f}$ and $[f\Gamma e] \neq 0$.\\
Let $c \in \mathcal{H}_{e} \cap \mathcal{I}_{f}$. then $e \mathcal{L} c$, $c \mathcal{R} f$ which implies $e \mathcal{D} f$.\\
Since, $(\mathcal{T},\Gamma)$ contains primitive idempotent say $a \neq 0$ which is in $\mathcal{D}$.\\
So, a is a regular element.\\
Since, $\mathcal{T}\setminus \{0\}$ is a $\mathcal{D}$-class.
By \textbf{Theorem}\ref{thm2.1}, $\mathcal{T}\setminus \{0\}$ is $\Gamma$-regular.\\
Therefore, we shown $(\mathcal{T},\Gamma)$ as a $\Gamma$-regular.
\end{proof}
\begin{example}
    Let $\mathcal{T}=\{i, -i, 1, -1\}$ and $\Gamma$=\{i\}. $(\mathcal{T}, \Gamma)$ under usual multiplication is completely 0-simple $\Gamma$-semigroup and is $\Gamma$-regular.
\end{example}
\begin{lemma}\label{lemma5.2}
If $(\mathcal{T},\Gamma)$ is a 0-2-sided simple $\Gamma$-semigroup containing a 0-least $\Gamma$-LId and a 0-least $\Gamma$-RId. Then, for each 0-least $\Gamma$-LId $\mathcal{H}$ of $(\mathcal{T},\Gamma)$, there is a 0-least $\Gamma$-RId $\mathcal{I}$, $[\mathcal{H}\Gamma \mathcal{I}]\neq 0$ and $[\mathcal{H}\Gamma \mathcal{I}]=\mathcal{T}$.
\end{lemma}
\begin{proof}
$[\mathcal{H}\Gamma \mathcal{T}]$ is a 2-Id of $(\mathcal{T},\Gamma)$.
Then, $[\mathcal{H}\Gamma \mathcal{T}]=0$ or $[\mathcal{H}\Gamma \mathcal{T}]=\mathcal{H}$.\\
Suppose $[\mathcal{H}\Gamma \mathcal{T}]=0$.
Then, $[\mathcal{H}\Gamma \mathcal{H}]=0$ and $\mathcal{H}$ is a 2-Id of $(\mathcal{T},\Gamma)$.\\
So, $\mathcal{T}=\mathcal{H}$.
Therefore, $[\mathcal{T}\Gamma \mathcal{T}]=[\mathcal{H}\Gamma \mathcal{T}]=0$ which is a contradiction to our assumption.\\
So, $[\mathcal{H}\Gamma \mathcal{T}]=\mathcal{H}$.\\
$[\mathcal{H}\Gamma x] \neq 0$ for some $x \in \mathcal{T}$.
By \textbf{Theorem}\ref{thm4.3}, $\mathcal{T}$ is the union of all 0-least $\Gamma$-RId of $\mathcal{T}$.\\
So, $[\mathcal{H}\Gamma \mathcal{I}]\neq 0$.\\
$[\mathcal{H}\Gamma \mathcal{I}]$ is the 2-Id of $(\mathcal{T},\Gamma)$.\\
Therefore, $[\mathcal{H}\Gamma \mathcal{I}]=\mathcal{T}$.
\end{proof}
\begin{lemma}\label{lemma5.3}
Consider $\mathcal{H}$ as  a 0-least LId of a 0-simple $\Gamma$-semigroup $(\mathcal{T},\Gamma)$ and $g \in \mathcal{H}\{0\}$. Then, $[\mathcal{H}\Gamma g]=\mathcal{H}$.
\end{lemma}
\begin{proof}
Since, $[\mathcal{T}\Gamma g]$ is a LId of $(\mathcal{T},\Gamma)$ included in $\mathcal{H}$.
So, $[\mathcal{T}\Gamma g]=0$ or $[\mathcal{T}\Gamma g]=\mathcal{H}$.\\
Since, $\mathcal{T}$ is 0-simple, by \textbf{Proposition}\ref{prop3.1}, $\mathcal{T}=\mathcal{T}\Gamma g \Gamma \mathcal{T}$.\\
Therefore, $[\mathcal{T}\Gamma g]\neq 0$.\\
Thus, $[\mathcal{T}\Gamma g]=\mathcal{H}$.
\end{proof}
\begin{lemma}\label{lemma5.4}
If $(\mathcal{T},\Gamma)$ is a  0-2-sided simple $\Gamma$-semigroup. Let $\mathcal{H}$ and $\mathcal{I}$ be a 0-least $\Gamma$-LId and a 0-least $\Gamma$-RId respectively with $[\mathcal{H}\Gamma \mathcal{I}]\neq 0$. Then, 
\begin{enumerate}
\item $[\mathcal{I}\Gamma \mathcal{H}]$ is a $\Gamma$-group with zero.
\item $[\mathcal{I}\Gamma \mathcal{H}]=\mathcal{I} \cap \mathcal{H}$.
\end{enumerate}
\end{lemma}
\begin{proof}
\begin{enumerate}
\item
Since, $(\mathcal{T},\Gamma)$ is a completely 0-2-sided simple. By \textbf{Proposition}\ref{prop3.1}, $\mathcal{T}=\mathcal{T}\Gamma g \Gamma \mathcal{T}$.\\
This implies $[g \Gamma \mathcal{T}]=\mathcal{I}$ and $[\mathcal{T}\Gamma g]=\mathcal{H}$.\\
By \textbf{Lemma}\ref{lemma5.2}, $[\mathcal{H}\Gamma \mathcal{I}]=\mathcal{T}$.\\
$\mathcal{T}=\mathcal{H}\Gamma g \Gamma \mathcal{T}$.
So, $[\mathcal{H}\Gamma g]\neq 0$ and $[\mathcal{H}\Gamma g]=\mathcal{H}$.\\
Therefore, $[\mathcal{I}\Gamma \mathcal{H}]\Gamma g=[\mathcal{I}\Gamma \mathcal{H}]$.\\
Thus, $[\mathcal{I}\Gamma \mathcal{H}]$ is $\Gamma$-left simple.\\
Similarly, we can show that $[\mathcal{I}\Gamma \mathcal{H}]$ is $\Gamma$-right simple.\\ 
So, $[\mathcal{I}\Gamma \mathcal{H}]$ is a $\Gamma$-group with zero by \textbf{Theorem 2.1} in \cite{sen1986semigroup}.
\item 
$[(\mathcal{I} \cap \mathcal{H})\Gamma g]=[\mathcal{I} \Gamma g] \cap [\mathcal{H} \Gamma g] = (\mathcal{I} \cap \mathcal{H})$.\\
So, $(\mathcal{I} \cap \mathcal{H})$ is $\Gamma$-left simple.\\
Similarly, we can show that $(\mathcal{I} \cap \mathcal{H})$ is $\Gamma$-right simple.\\ 
So, $(\mathcal{I} \cap \mathcal{H})$ is a $\Gamma$-group with zero.\\
Let $g \in (\mathcal{I} \cap \mathcal{H})$.
Then, $g=g \Gamma e \in \mathcal{I}\Gamma \mathcal{H}$.
So, $(\mathcal{I} \cap \mathcal{H}) \cap [\mathcal{I}\Gamma \mathcal{H}]$.\\
Clearly, $[\mathcal{I}\Gamma \mathcal{H}] \cap (\mathcal{I} \cap \mathcal{H})$,\\
Therefore, $[\mathcal{I}\Gamma \mathcal{H}] = (\mathcal{I} \cap \mathcal{H})$
\end{enumerate}
\end{proof}
\begin{lemma}\label{lemma5.5}
If $(\mathcal{T},\Gamma)$ is a  0-2-sided simple $\Gamma$-semigroup. Let $\mathcal{H}$ and $\mathcal{I}$ be a 0-least $\Gamma$-LId and a 0-least $\Gamma$-RId respectively with $[\mathcal{H}\Gamma \mathcal{I}]\neq 0$. Let e be the identity element of $\mathcal{H}$ and $\mathcal{I}$. Then, 
\begin{enumerate}
\item $\mathcal{I}=[e\Gamma \mathcal{T}]$, $\mathcal{H}=[\mathcal{T}\Gamma e]$ and $[\mathcal{I}\Gamma \mathcal{H}]=[e\Gamma \mathcal{T} \Gamma e]$.
\item e is the primitive idempotent of $(\mathcal{T},\Gamma)$.
\end{enumerate}
\end{lemma}
\begin{proof}
\begin{enumerate}
\item 
$e \in \mathcal{H}\{0\}$.\\
By \textbf{Lemma}\ref{lemma5.3}, $[\mathcal{T}\Gamma e]=\mathcal{H}$.\\
Similarly, $[e \Gamma \mathcal{T}]=\mathcal{I}$.\\
$[\mathcal{I}\Gamma \mathcal{H}]=[e\Gamma \mathcal{T} \Gamma \mathcal{T} \Gamma e]=[e\Gamma \mathcal{T} \Gamma e]$.
\item 
Let f be an idempotent of $\mathcal{T}$ such that $f \geq e$.\\
$[fxe]=[eyf]=e$ for all $x,y \in \Gamma$.\\
So, $f \in [e\Gamma \mathcal{T} \Gamma e]$.\\
But, $[e\Gamma \mathcal{T} \Gamma e]=[\mathcal{I}\Gamma \mathcal{H}]$.\\
$[\mathcal{I}\Gamma \mathcal{H}]$ is a $\Gamma$-group with zero.\\
The only idempotents in a group are identity and zero.
So, $f=e$ and $f=0$.
Hence, e is primitive.
\end{enumerate}
\end{proof}
\begin{lemma}\label{lemma5.6}
If $(\mathcal{T},\Gamma)$ is a completely 0-2-sided simple and e is a primitive idempotent of $(\mathcal{T},\Gamma)$. Then, $\mathcal{H}=[\mathcal{T}\Gamma e]$ and $\mathcal{I}=e \Gamma \mathcal{T}$ are 0-least $\Gamma$-LIds and 0-least $\Gamma$-RIds respectively and $[\mathcal{I}\Gamma \mathcal{H}]$ is a $\Gamma$-group with zero with e as the identity.
\end{lemma}
\begin{proof}
We have to show that $\mathcal{I}=[e \Gamma \mathcal{T}]$ are 0-least $\Gamma$-RId.\\
$\mathcal{I} \neq 0$ since, $e \in \mathcal{I}$.
Assume $\mathcal{M}$ be a RId which is in $\mathcal{I}$.\\
Let $m \in \mathcal{M}\ \{0\}$.\\
Then, $m \in [e \Gamma \mathcal{T}]$, there is a $t \in \Gamma$ such that 
\begin{equation}
[etm]=m
\end{equation}
Since, $\mathcal{T}$ is 0-simple and $m \neq 0$, by \textbf{Proposition}\ref{prop3.1}, $\mathcal{T}=\mathcal{T}\Gamma m \Gamma \mathcal{T}$.\\
Then, there exist $x^{'},y^{'} \in \mathcal{T}, p^{'},q^{'} \in \Gamma$, $x^{'}p^{'}mq^{'}y^{'}=e$.\\
e is an idempotent and so, there exist $t_{1} \in \mathcal{T}$ such that 
\begin{equation}
et_{1}e=e
\end{equation}
Let $x=et_{1}x^{'}p^{'}e$ and $y=y^{'}t_{1}e$.\\
Then, 
\begin{equation}
xtmq^{'}y=e.
\end{equation}
Take $f=mq^{'}yt_{1}xte$.\\
\begin{align}
\notag
et_{1}x&=et_{1}et_{1}x^{'}p^{'}e \\
\notag
&=et_{1}x^{'}p^{'}e \\
&=
\begin{aligned}
x
\end{aligned}
\end{align}
\begin{align*}
ftf &= (mq^{'}yt_{1}xte)t(mq^{'}yt_{1}xte) \\
&=mq^{'}yt_{1}(xtetmq^{'}y)t_{1}xte\ [By\ eq.(3)]\\
&=mq^{'}yt_{1}et_{1}xte\ [By\ eq.(4)]\\ 
&=mq^{'}yt_{1}xte \\
&=f.
\end{align*}
\begin{align*}
etf&=etmq^{'}yt_{1}xte \\
&=mq^{'}yt_{1}xte \ [By\ eq.(1)] \\
&=f.
\end{align*}
\begin{align*}
ft_{1}e&=mq^{'}yt_{1}xtet_{1}e \ [By \ eq.(2)]\\
&=f.
\end{align*}
Now,
\begin{align*}
e&=et_{1}e \ [By \ eq.(2)]\\
&=xtmq^{'}yt_{1}xtmq^{'}y \ [By \ eq.(3)]\\
&=xtmq^{'}yt_{1}xtetmq^{'}y \ [By \ eq.(1)]\\
&=xtftmq^{'}y
\end{align*}
This implies $f \neq 0$.\\
Thus, f is a non zero idempotent.\\
e is an idempotent implies $f=e$.\\
$e=mq^{'}yt_{1}xte \in m\Gamma \mathcal{T}$.\\
Therefore, $\mathcal{I}=e\Gamma \mathcal{T} \subseteq m\Gamma \mathcal{T} \Gamma \mathcal{T}=m\Gamma \mathcal{T}$.\\
Thus, $\mathcal{M}=\mathcal{I}$ and $\mathcal{I}$ is a 0-least $\Gamma$-RId.\\
Similarly, we can show that $\mathcal{H}$ is a 0-least $\Gamma$-LId.\\
By the \textbf{Lemma}\ref{lemma5.4}, $[\mathcal{I}\Gamma \mathcal{H}=[e\Gamma \mathcal{T} \Gamma e]$ is a $\Gamma$-group with zero.\\
$e \in [e\Gamma \mathcal{T} \Gamma e]=[\mathcal{I}\Gamma \mathcal{H}$ and $e \neq 0$.\\
So, e is the identity element of $[\mathcal{I}\Gamma \mathcal{H}]$.
\end{proof}
\begin{theorem}\label{thm5.2}
Let $(\mathcal{T},\Gamma)$ is a 0-2-sided simple. Then, the necessary and sufficient conditions for $(\mathcal{T},\Gamma)$ to be completely 0-2-sided simple is it contains atleast one 0-least $\Gamma$-LId and one 0-least $\Gamma$-RId.
\end{theorem}
\begin{proof}
Assume $(\mathcal{T},\Gamma)$ to be completely 0-2-sided simple.\\
It contains primitive idempotent e.\\
By \textbf{Lemma}\ref{lemma5.6}, $\mathcal{H}=[\mathcal{T}\Gamma e]$ and $\mathcal{I}=e \Gamma \mathcal{T}$ are 0-least $\Gamma$-LIds and 0-least $\Gamma$-RIds respectively.\\
Conversly assume $(\mathcal{T},\Gamma)$ contains atleast one 0-least $\Gamma$-LId and one 0-least $\Gamma$-RId.\\
Let $\mathcal{H}$ be a 0-least $\Gamma$-LId of $(\mathcal{T},\Gamma)$.\\
By \textbf{Lemma}\ref{lemma5.2}, there is a 0-least $\Gamma$-RId $\mathcal{I}$, $[\mathcal{H}\Gamma \mathcal{I}] \neq 0$.\\
By \textbf{Lemma}\ref{lemma5.5}, $(\mathcal{T},\Gamma)$ contains a primitive idempotent.\\
Therefore, $(\mathcal{T},\Gamma)$ is a completely 0-2-sided simple.
\end{proof}
\begin{corollary}
A complete 0-simple $\Gamma$-semigroup is the union of its 0-least $\Gamma$-LIds(RIds).
\end{corollary}
\begin{proof}
The proof is immediate from \textbf{Theorem}\ref{thm4.3} and \textbf{Theorem}\ref{thm5.2}.
\end{proof}
\section{$\Gamma$-Prime Ideals}
\begin{definition}
    Let $(\mathcal{T},\Gamma)$ be a $\Gamma$-Semigroup. A $\Gamma$-2-Id $\mathcal{Q}$ is said to be $\Gamma$-prime if for any $\Gamma$-2-Id $\mathcal{E}, \mathcal{F} \in \mathcal{T}$, $[\mathcal{E}\Gamma \mathcal{F}] \subseteq \mathcal{Q}$ implies $\mathcal{E} \subseteq \mathcal{Q}$ or $\mathcal{F} \subseteq \mathcal{Q}$.
\end{definition}
\begin{example}
    Consider \[\mathcal{T} = \{\text{Set of all integers}\}\] \[\Gamma = \{\text{Set of all integers of the form } 2n+1 \text{ where } n \in \mathbb{N}\}\]. Then, \[\mathcal{E} = \{\text{Set of all integers of the form } 2n \text{ where } n \in \mathbb{N}\}\] is a $\Gamma$-prime ideal.
\end{example}
The following gives the characterization of $\Gamma$-prime ideals.
\begin{theorem}\label{thm 6.1}
     Let $(\mathcal{T},\Gamma)$ be a $\Gamma$-Semigroup. Then, a $\Gamma$-2-Id $\mathcal{Q}$ is $\Gamma$-prime if and only if for any $e, f \in \mathcal{T}$, $\{e\}_{2} \Gamma \{f\}_{2} \subseteq \mathcal{Q}$ implies $e \in \mathcal{Q}$ or $f \in \mathcal{Q}$ where, $\{e\}_{2}, \{f\}_{2}$ are the $\Gamma$-2-Id generated by e, f respectively.
\end{theorem}
The theorem below gives the useful result for a commutative $\Gamma$-semigroup to be $\Gamma$-prime.
\begin{theorem}
    Let $(\mathcal{T},\Gamma)$ be a commutative $\Gamma$-Semigroup. Then, a $\Gamma$-2-Id $\mathcal{Q}$ is $\Gamma$-prime if and only if for any $e, f \in \mathcal{T}$, $e \Gamma f \subseteq \mathcal{Q}$ implies $e \in \mathcal{Q}$ or $f \in \mathcal{Q}$
\end{theorem} 
\begin{proof}
    Assume $\mathcal{Q}$ is $\Gamma$-prime. Let $e, f \in \mathcal{T}$ and $e \Gamma f \subseteq \mathcal{Q}$.
    If $\mathcal{Q}$ is $\Gamma$-prime, then we have for any $e, f \in \mathcal{T}$, $\{e\}_{2} \Gamma \{f\}_{2} \subseteq \mathcal{Q}$ implies $e \in \mathcal{Q}$ or $f \in \mathcal{Q}$ where, $\{e\}_{2}, \{f\}_{2}$ are the $\Gamma$-2-Id generated by e, f respectively by \textbf{Theorem}\ref{thm 6.1}. 
    \begin{align*}
        \{e\}_{2} \Gamma \{f\}_{2} = ([\mathcal{T} \Gamma e \Gamma \mathcal{T}] \cup \{e\}) \Gamma ([\mathcal{T} \Gamma f \Gamma \mathcal{T}] \cup \{f\}) &\\ = [\mathcal{T} \Gamma e \Gamma \mathcal{T}]  \Gamma [\mathcal{T} \Gamma f \Gamma \mathcal{T}] \cup [e \Gamma f] &\\ \subseteq \mathcal{T} \Gamma e \Gamma \mathcal{T} \Gamma f \Gamma \mathcal{T} \cup \mathcal{Q} &\\ \subseteq e \Gamma \mathcal{T} \Gamma \mathcal{T} \Gamma \mathcal{T} \Gamma  f \cup \mathcal{Q} &\\ \subseteq e \Gamma f \Gamma \mathcal{T} \cup \mathcal{Q} &\\ \subseteq \mathcal{Q} \Gamma \mathcal{T} \cup \mathcal{Q} &\\ \subseteq \mathcal{Q}
    \end{align*} 
    This implies $e \in \mathcal{Q}$ or $f \in \mathcal{Q}$.\\
    Conversly, suppose that for any $e, f \in \mathcal{T}$, $e \Gamma f \subseteq \mathcal{Q}$ implies $e \in \mathcal{Q}$ or $f \in \mathcal{Q}$.
    Suppose for any $\mathbf{E}, \mathcal{F} \in \mathcal{T}$, $\mathbf{E} \Gamma \mathcal{F} \subseteq \mathcal{Q}$ and $\mathcal{F} \not\subseteq \mathcal{Q}$.
    Then, there exist a $f \in \mathcal{F}$ such that $f \not\in \mathcal{Q}$. For any $e \in \mathcal{E}$, $e \Gamma f \subseteq \mathcal{E} \Gamma \mathcal{F} \subseteq \mathcal{Q}$. Hence, $\mathcal{E} \in \mathcal{Q}$.
    Therefore, $\mathcal{Q}$ is a $\Gamma$-prime.
    
\end{proof}
\begin{example}
    Consider \[ \mathcal{T}=\{Set\ of\ all\ negative\ integers\}\], \[ \Gamma=\{Set\ of\ all\ negative\ integers\ of \ the \ form \ 2n+1 \ where,\ n \in \mathbf{N}\ \}\]
    \[ \mathcal{Q}_{1}=\{ 2p:p\in \mathcal{T} \} \] is $\Gamma$-prime.

    \[
\begin{array}{l}
e \Gamma f \in \mathcal{Q}_{1} \quad (e, f \in \mathcal{T}) \\
\Longleftrightarrow e \Gamma f \text{ is divisible by } 3 \\
\Longleftrightarrow e \text{ is divisible by } 3 \text{ or } f \text{ is divisible by } 3 \\
\Longleftrightarrow e \in \mathcal{Q}_{1} \text{ or } f \in \mathcal{Q}_{1}.
\end{array}
\]

    But, \[ \mathcal{Q}_{2}=\{ 20p:p\in \mathcal{T} \} \] is not $\Gamma$-prime. Since, $ -2 \times -5 \times -2 \in \mathcal{Q}_{2} $ but, $-2 \not \in \mathcal{Q}_{2}$.
\end{example}
\begin{theorem}
    If $\mathcal{H}$ is a $\Gamma$-2-Id of $\Gamma$-semigroup $(\mathcal{T}, \Gamma)$ and $\mathcal{Q}$ is a $\Gamma$-prime ideal, then $\mathcal{H} \cap \mathcal{Q}$ is a $\Gamma$-prime ideal.
\end{theorem}
\begin{proof}
    Let $e, f \in \mathcal{H}$. 
    \[ \{e\}_{2} \Gamma \{f\}_{2} \subseteq \mathcal{H} \cap \mathcal{Q} \subseteq \mathcal{Q}\]
    From \textbf{Theorem}\ref{thm 6.1}, $e \in \mathcal{Q}$ or $f \in \mathcal{Q}$. This implies $e \in \mathcal{H} \cap \mathcal{Q}$ or $f \in \mathcal{H} \cap \mathcal{Q}$.
    Clearly, $\mathcal{H} \cap \mathcal{Q}$ is a $\Gamma$-2-Id of $\mathcal{H}$.
    Therefore, from \textbf{Theorem}\ref{thm 6.1}, $\mathcal{H} \cap \mathcal{Q}$ is a $\Gamma$-prime ideal of $\mathcal{H}$.
\end{proof}
\begin{example}
    Consider \[ \mathcal{T}=\{Set\ of\ all\ negative\ integers\}\], \[ \Gamma=\{Set\ of\ all\ negative\ integers\ of \ the \ form \ 2n+1 \ where,\ n \in \mathbf{N}\ \}\]
    \[ \mathcal{Q}_{1}=\{ 2p:p\in \mathcal{T} \} \] and \[ \mathcal{Q}_{2}=\{ 3p:p\in \mathcal{T} \} \] are $\Gamma$-prime. But, $\mathcal{Q}_{1} \cup \mathcal{Q}_{2}$ is not $\Gamma$-prime.
\end{example}
\textbf{Note:}Union and intersection of $\Gamma$-prime ideals need not be $\Gamma$-prime.

\begin{theorem}
    Let $\mathcal{Q}_{i}$ be a collection of $\Gamma$-prime ideals of a $\Gamma$-semigroup $(\mathcal{T}, \Gamma)$ such that $\mathcal{Q}_{i}$ forms a dcc or ACC. Then, $\bigcup {\mathcal{Q}_{i}} $ and $\bigcap {\mathcal{Q}_{i}} $ are $\Gamma$-prime.
\end{theorem}
\begin{proof}
    From proposition, $\bigcup {\mathcal{Q}_{i}} $ is a $\Gamma$-2-Id.
    Let $\mathcal{E} \Gamma \mathcal{F} \subseteq \bigcup {\mathcal{Q}_{i}}$ for any $\Gamma$-2-Id $\mathcal{E}$ and $\mathcal{F}$. We have to show that $\mathcal{E}  \subseteq \bigcup {\mathcal{Q}_{i}}$ or $ \mathcal{F} \subseteq \bigcup {\mathcal{Q}_{i}}$. \\
    If there is an $i$ such that $\mathcal{E}  \subseteq  {\mathcal{Q}_{i}}$ or $ \mathcal{F} \subseteq  {\mathcal{Q}_{i}}$, then $\mathcal{E}  \subseteq \bigcup {\mathcal{Q}_{i}}$ or $ \mathcal{F} \subseteq \bigcup {\mathcal{Q}_{i}}$.
    \\ Otherwise, Suppose $\mathcal{E} \Gamma \mathcal{F} \not \subseteq \bigcup {\mathcal{Q}_{i}}$. Then, there exist $i,j$ such that $\mathcal{E}  \subseteq  {\mathcal{Q}_{i}}$ or $ \mathcal{F} \subseteq  {\mathcal{Q}_{j}}$. Since, ${\mathcal{Q}_{i}}$ satisfies ACC or dcc, Take ${\mathcal{Q}_{i}} \subseteq \mathcal{Q}_{j} $. This implies $ \mathcal{F} \not \subseteq  {\mathcal{Q}_{i}}$. Since, $\mathcal{E} \Gamma \mathcal{F} \subseteq  {\mathcal{Q}_{i}}$ and $\mathcal{Q}_{i}$ is a $\Gamma$-prime.
    Then, $\mathcal{E}  \subseteq  {\mathcal{Q}_{i}}$ or $ \mathcal{F} \subseteq  {\mathcal{Q}_{i}}$, a contradiction.
    Therefore, $\mathcal{E}  \subseteq  {\mathcal{Q}_{i}}$ or $ \mathcal{F} \subseteq  {\mathcal{Q}_{i}}$, then $\mathcal{E}  \subseteq \bigcup {\mathcal{Q}_{i}}$ or $ \mathcal{F} \subseteq \bigcup {\mathcal{Q}_{i}}$. Therefore, $\bigcup {\mathcal{Q}_{i}}$ is a $\Gamma$-prime.\\
    Similarly, $\bigcap {\mathcal{Q}_{i}}$ can be proved.
    
\end{proof}
\section{Conclusion}
 We have explored the concept of $\Gamma$-semigroups in depth. Section 2 provided the fundamental definitions essential for understanding the subsequent sections. In Section 3, we introduced the notion of $\Gamma$-simple semigroups, outlined their fundamental properties, and defined $\Gamma$-o-simple semigroups, offering characterizations and related results.

Section 4 delved into the concept of o-least $\Gamma$-ideals, presenting key theorems concerning $\Gamma$-LIds. In Section 5, we investigated the concept of completely 0-2-sided least simple $\Gamma$-semigroups, demonstrating that such a semigroup is its own $\mathcal{D}$-class and exhibits regularity. Also, we established that the presence of at least one 0-least $\Gamma$-left ideal (0-least $\Gamma$-LId) and one 0-least $\Gamma$-right ideal (0-least $\Gamma$-RId) is both a necessary and sufficient condition for a $(\mathcal{T},\Gamma)$ to be completely 0-simple.
In this work, we explored the concept of $\Gamma$-prime ideals within the framework of $\Gamma$-semigroups. We identified conditions under which a $\Gamma$-2-sided ideal can be classified as a $\Gamma$-prime ideal, providing concrete examples to illustrate these results. For commutative $\Gamma$-semigroups, we established a sufficient condition to ensure the $\Gamma$-prime property, contributing to the understanding of their structural characteristics.

Counterexamples demonstrated that the union and intersection of $\Gamma$-prime ideals are not necessarily $\Gamma$-prime, highlighting potential complexities in their behavior. However, we formulated and proved conditions under which the union and intersection of $\Gamma$-prime ideals becomes $\Gamma$-prime.

\section*{Acknowledgements}
We would like to express our sincere gratitude to our institution for giving support, which was invaluable in the completion of this study.

\noindent\hrulefill

\end{document}